\documentclass[12pt]{amsart}

\usepackage{newpxtext,amssymb,latexsym,amsmath,amsfonts,amsthm,amscd,graphicx,url,color,hyperref, eulervm}

\usepackage{hyperref}
\hypersetup{
 colorlinks,
 linkcolor={blue!90!black},
 citecolor={red!80!black},
 urlcolor={blue!50!black},
breaklinks=true
}
\usepackage{tabularx}
\usepackage{comment}
\usepackage{amscd}
\usepackage{blkarray}
\usepackage{booktabs}
\usepackage{enumitem}
\usepackage{aliascnt}
\usepackage[initials, lite]{amsrefs}
\usepackage[capitalize,nameinlink]{cleveref}
\usepackage{tikz, tikz-cd}
\usepackage{calligra,mathrsfs}
\usetikzlibrary{matrix}
\usetikzlibrary{arrows,calc}
\usepackage{url}

\allowdisplaybreaks

\theoremstyle{plain}

\newtheorem{thm}{Theorem}[section]
\newtheorem{prop}[thm]{Proposition}
\newtheorem{cor}[thm]{Corollary}
\newtheorem{lemma}[thm]{Lemma}
\newtheorem{conj}[thm]{Conjecture}

\theoremstyle{definition}
\newtheorem{definition}[thm]{Definition}

\newtheorem{remark}[thm]{Remark}
\newtheorem{question}[thm]{Question}
\newtheorem{examples}[thm]{Examples}
\newtheorem{notation}[thm]{Notation}
\newtheorem{convention}[thm]{Convention}
\newtheorem{example}[thm]{Example}
\numberwithin{equation}{section}

\newcommand{\Proj}{\mathrm{Proj}}
\newcommand{\Hilb}{\mathrm{Hilb}}
\newcommand{\Tor}{\mathrm{Tor}}
\newcommand{\HF}{\mathrm{HF}}
\newcommand{\HP}{\mathrm{HP}}

\newcommand{\NN}{\mathbb{N}}
\newcommand{\PP}{\mathbb{P}}
\newcommand{\ZZ}{\mathbb{Z}}

\newcommand{\Lex}{{\mathrm{Lex}}}

\newcommand{\ch}{\mathrm{char}}
\newcommand{\depth}{\mathrm{depth}}

\newcommand{\rank}{\mathrm{rank}}
\newcommand{\lex}{{\mathrm{lex}}}
\newcommand{\ann}{{\mathrm{ann}}}
\renewcommand{\deg}{{\mathrm{deg}}}

\newcommand{\mult}{\mathrm{mult}}

\newcommand{\mm}{\mathfrak{m}}

\newcommand{\bfu}{\mathbf{u}}
\newcommand{\bfv}{\mathbf{v}}
\newcommand{\bfe}{\mathbf{e}}

\newcommand{\ovS}{{\overline{S}}}
\newcommand{\ovR}{{\overline{R}}}

\newcommand{\wS}{{\widetilde{S}}}
\newcommand{\wR}{{\widetilde{R}}}

\newcommand{\wI}{{\widetilde{I}}}
\newcommand{\wJ}{{\widetilde{J}}}
\newcommand{\wL}{{\widetilde{L}}}

\newcommand{\DD}{{\mathbb{D}}}
\newcommand{\EE}{{\mathbb{E}}}
\newcommand{\FF}{{\mathbb{F}}}

\newcommand{\CC}{\mathfrak{C}}

\newcommand{\mI}{{\mathcal{I}}}

\newcommand{\mC}{{\mathcal{C}}}
\newcommand{\mJ}{{\mathcal{J}}}

\makeatletter
\@namedef{subjclassname@2020}{
  \textup{2020} Mathematics Subject Classification}
\makeatother

\title{Syzygies in Hilbert schemes of complete intersections}

\author{Giulio Caviglia \and Alessio Sammartano}
\address[Giulio Caviglia]{Department of Mathematics, Purdue University, 
West Lafayette,
IN, USA}
\email{gcavigli@purdue.edu}
\address[Alessio Sammartano]{Dipartimento di Matematica, Politecnico di Milano, 
 Milano, Italy}
\email{alessio.sammartano@polimi.it}
\thanks{The work of the first named author was partially supported by a grant from the Simons Foundation (41000748, G.C.).
The second named author is a member of the Gruppo Nazionale per le Strutture Algebriche, Geometriche e le loro Applicazioni (GNSAGA)  of INdAM, and was partially supported by 
 PRIN 2020355B8Y “Squarefree Gr\"oner degenerations, special varieties and related topics”. }

\begin{document}

\begin{abstract}
Let $ e_1, \ldots,  e_{c}  $ be positive integers
and let $ Y \subseteq \PP^n$ be the monomial complete intersection defined by the vanishing of $x_1^{e_1}, \ldots, x_{c}^{e_{c}}$.
In this paper,
 we study sharp upper  bounds on the number of equations and syzygies of subschemes parametrized by the Hilbert scheme of points $\Hilb^d(Y)$, 
and discuss applications to the Hilbert scheme of points $\Hilb^d(X)$ of arbitrary complete intersections $X \subseteq \PP^n$.
\end{abstract}

\maketitle

\section*{Introduction}

In this paper,
 we investigate the extremal behavior of free resolutions 
of finite subschemes of complete intersections $X \subseteq \PP^n$.
Our motivating question is the following.
Let $\bfe = (e_1, \ldots, e_c)$ be a degree sequence and $d$ a positive integer: 
are there uniform bounds on the syzygies of $Z\subseteq X$, where $X\subseteq \PP^n$ is a complete intersection of degrees $\bfe$ and $Z\subseteq X$ a finite subscheme of length $d$?

In order to address this problem, we study Hilbert schemes of points of Clements-Lindstr\"om schemes $Y \subseteq \PP^n$, 
 defined by the vanishing of pure powers $x_1^{e_1}, \ldots, x_{c}^{e_{c}}$.
Our  main result, Theorem \ref{TheoremExtremalFiniteResolution}, states that a distinguished monomial ideal 
$\CC(d) $ attains the largest possible number of $i$-th syzygies for a subscheme in $\Hilb^d(Y)$, for every homological degree $i$. 
There are advantages in  considering   Clements-Lindstr\"om schemes $Y$ for various degree sequences $\bfe$, as opposed to just considering $\PP^n$.
First, 
by taking the degree sequence into account, and restricting thus to a smaller Hilbert scheme,
one obtains sharper numerical bounds on Betti numbers.
A similar point of view is adopted e.g. in  \cite{EiGrHa93}, where 
 bounds on the  number of points in intersections of quadric hypersurfaces are improved using the data of the degree sequence,
 or  in the study of balanced Cohen-Macaulay simplicial complexes in \cite{JKV18}.
More importantly, 
our bounds extend conjecturally to arbitrary complete intersections in $\PP^n$.
In fact, we show that, under the validity of the  Lex Plus Powers Conjecture, the distinguished ideal $\CC(d)$ yields uniform bounds for the syzygies of  subschemes $Z\in \Hilb^d(X)$ for \emph{all} complete intersections $X\subseteq \PP^n$ of degrees $\bfe$,
thus giving a complete answer to our motivating problem.

When restricting to $c=0$, that is, for $Y =\PP^n$, we recover the main result of \cite{Va94}. 
In fact, a major motivation for this work was  the desire to extend classical results on $\Hilb^d(\PP^n)$  \cite{Be81,BrIa78,ElRoVa91,Va94}
 to  the general setting of complete intersections.

We  apply our methods also  to infinite free resolutions over complete intersections, motivated by the recent progress in this area \cite{EiPe16,EiPe20}.
We conjecture that the extremal behavior of Theorem \ref{TheoremExtremalFiniteResolution} holds for  infinite free resolutions, and prove this conjecture for quadratic Clements-Lindstr\"om rings in characteristic zero in Theorem \ref{TheoremIfninite}.
We also discuss the analogous problem for  deviations and Poincar\'e series, 
solving it in the case of $\Hilb^d(\PP^n)$ in Corollaries \ref{CorollaryDeviations} and \ref{CorollaryPoincare}.

\subsection*{Organization}
In Section \ref{SectionPreliminaries}, 
we set up the notation, introduce Clements-Lindstr\"om rings and the relevant classes of monomial ideals.
Section \ref{SectionDecomposition} discusses a decomposition of monomial ideals in Clements-Lindstr\"om rings, which plays an important role in the recursive study of syzygies.
In Section \ref{SectionSyzygies},
 we prove our main result on extremal syzgyies in $\Hilb^d(Y)$, Theorem \ref{TheoremExtremalFiniteResolution}. 
It relies on the study of the decomposition of a special monomial ideal $\CC(d)$, which we carry out in detail in Proposition \ref{PropositionMainPropertiesCd},
and in particular on the extremality of $\CC(d)$ with respect to certain ``hypersurface sections''.
In Section \ref{SectionInfinite},
 we study infinite free resolutions.
Our main result is Theorem \ref{TheoremIfninite},
where we combine the tools of Section \ref{SectionSyzygies} with a construction of \cite{ArAvHe00,EiPoYu03,GaHiPe02}
to show that the ideal $\CC(d)$ also attains the maximum Betti numbers of the infinite free resolution over quadratic Clements-Lindstr\"om rings, if the ground field has characteristic zero.
Finally, in Section \ref{SectionApplications},
 we conclude the paper with some applications to arbitrary complete intersections, combining our main results with the known cases of the
Eisenbud-Green-Harris and Lex-Plus-Powers conjectures.

\section{Clements-Lindstr\"om rings}\label{SectionPreliminaries}

Let $\NN$ denote the set of nonnegative  integers,
and let $\Bbbk$ denote an arbitrary field. 
All  rings considered in this work are standard graded $\Bbbk$-algebras, and all ideals and modules are graded;
these attributes are often assumed implicitly and  omitted.

Let  $V$ be a $\ZZ$-graded $\Bbbk$-vector space.
The $j$-th graded component is denoted by $[V]_j$.
The  \emph{Hilbert function}
$\HF(V) $ is  $\HF(V,j) = \dim_\Bbbk[V]_j$.
The \emph{Hilbert polynomial} $\HP(V)$, when it exists, satisfies $\HP(V,j) = \HF(V,j)$  for all $j \gg 0$.

Let $A$ be a ring and $I\subseteq A$ an ideal.
The  maximal ideal of   $A$ is denoted by $\mm_A$.
An ideal $I\subseteq A$ is \emph{saturated} if $I : \mm_A = I$, 
equivalently, if $\depth(A/I)>0$.
The saturation $I : \mm_A^\infty = \cup_{t \geq 0} I : \mm_A^t$ of $I \subseteq A $
 is a saturated ideal with 
$\HP(I : \mm_A^\infty) =\HP(I)$.
If $M$ is a finite  $A$-module,
the integers 
$$
\beta_{i,j}^A(M) = \dim_\Bbbk[\Tor^A_i(M,\Bbbk)]_j
\quad \text{and} \quad  
\beta_{i}^A(M) = \dim_\Bbbk\Tor^A_i(M,\Bbbk)$$
are  the {\it graded Betti numbers} and the {\it (total) Betti numbers} of $M$, respectively.

Let $I \subseteq A$ be an ideal.
The \emph{multiplicity} of $A/I$, 
defined as  normalized leading coefficient of $\HP(A/I)$, is denoted by $\mult(I)$.
This slight abuse of notation should not generate confusion, since the multiplicity of $I$ as $A$-module often does not carry interesting information.
When $\dim(A/I)=1$, 
as in the setting of this paper, 
$\HP(A/I)$ is a constant polynomial, equal to 
$\mult(I)$.
When $\dim(A/I)=\depth(A/I)=1$,
then  $\mult(I) = \dim_\Bbbk\frac{A}{I+(z)}$ 
where $z\in[A]_1$ is a non-zerodivisor on $A/I$.

Given the projective scheme $X= \Proj A$ and   $d\in\NN$,
the {\it Hilbert scheme of points}, denoted by $\Hilb^d(X)$, is the projective scheme parametrizing  finite subschemes  $Z\subseteq X$ of length $d$, 
equivalently, with $\mult(I_Z) = d$.
As it is common in the literature, 
we    identify a closed subscheme $Z \subseteq X$ with its saturated  ideal $I_Z \subseteq A$ and with the point on the Hilbert scheme parametrizing it.
Moreover, we adopt the following:

\begin{convention}
If $I\subseteq A$ is an ideal, the expression  ``$I\in \Hilb^d(\Proj A)$'' means that $I$ is saturated,
$\dim(A/I) = 1$ and $\mathrm{mult}(I)=d$.
\end{convention}

We now introduce the rings that are central to this work.

\begin{convention}\label{ConventionInfinity}
We will  use  $\NN\cup \{\infty\}$ as index set and as range for exponents. 
We adopt standard conventions on $\infty$, e.g. $\ell < \infty $  and  $\infty-\ell = \infty$ for all $\ell \in \NN$.
If $r$ is an element in a ring, then  $r^\infty := 0$.
If $e = \infty$, the expression ``$\ell < e$''   means ``$\ell\in \NN$''.
\end{convention}

\begin{definition}
A {\it Clements-Lindstr\"om ring} is a ring of the form 
$$
A = \frac{\Bbbk [x_1, \ldots, x_{m}]}{\big(x_1^{e_1}, \ldots, x_m^{e_{m}}\big)}
$$
for some  $e_1 \leq e_2 \leq \cdots \leq e_{m}$ with $e_i \in \NN\cup \{\infty\}$.
\end{definition}

For the remainder of this section,   $A$ denotes a Clements-Lindstr\"om ring.
We emphasize that $x_i^{\infty}=0$;
thus, $A$ is a polynomial ring if $e_1 = \infty$, whereas it is Artinian if $e_m < \infty$.

\begin{remark}
Suppose that $\Proj A \ne \emptyset$, equivalently, $e_m = \infty$.
Then, $\Hilb^d(\Proj A) \ne 0$ if and only if 
 either $\dim A > 1$, that is, $e_{m-1} = \infty $, or 
$\dim A =1$ and the multiplicity of $A$ is at least $d$,
that is, 
$ e_{m-1} < \infty$ and $d \leq e_1 e_2 \cdots e_{m-1}$.
\end{remark}

An ideal $I \subseteq A$ is monomial if it is the image of a monomial ideal of $\Bbbk [x_1, \ldots, x_{m}]$. 
We denote the lexicographic monomial order in  $A$ by $<_\lex$.
A monomial  ideal $I \subseteq A$ is {\it lex} if every $[I]_j$ is a generated by an initial segment with respect to $<_\lex$,
equivalently, if $I$ is the image of a lex ideal of $\Bbbk [x_1, \ldots, x_{m}]$.
The saturation of a lex ideal is again lex.
A theorem of Clements and Lindstr\"om, which generalizes the classical results of Macaulay and Kruskal-Katona,
 states that lex ideals classify Hilbert functions in $A$:

\begin{prop}[\cite{ClLi69}]\label{PropositionCL}
Let $A$ be a Clements-Lindstr\"om ring and $I\subseteq A$ an ideal.
There exists a unique lex ideal $L \subseteq A$ such that $\HF(L) = \HF(I)$.
\end{prop}

If $\mathcal{H}$ is the Hilbert function of some ideal of $A$, we denote by $\Lex(\mathcal{H},A)$  the  lex ideal $L\subseteq A $ with $\HF(L) = \mathcal{H}$, and, 
if $I \subseteq A$, we define  $\Lex(I):=\Lex(\HF(I), A)$. 

A monomial ideal $I \subseteq A$ is {\it strongly stable} if we have 
$\frac{x_k \bfu}{x_h}\in I$ whenever  $\bfu \in I$ is a  nonzero monomial,  $x_h $ divides $\bfu$, and $k <h$.
It suffices to check this condition for the  generators $\bfu$ of $I$.
A strongly stable ideal $I \subseteq A$ is saturated if and only if the last variable
$x_{m}$ is a non-zerodivisor on $A/I$;
when $\dim A >0$, this is equivalent to $x_{m}$ not dividing any monomial  generator of $I$.
The saturation of a strongly stable ideal is again strongly stable.

A monomial ideal $I \subseteq A $ is  {\it almost lex} if  the last variable $x_{m}$ is a non-zerodivisor on $A/I$  and
 $({I+(x_{m})})/{(x_{m})}$ is a lex ideal of  the Clements-Lindstr\"om ring ${A}/{(x_{m})} 
\cong {\Bbbk [x_1, \ldots, x_{m-1}]}/{(x_1^{e_1}, \ldots, x_{m-1}^{e_{m-1}})}$.
Thus, almost lex ideals are saturated.
Observe that a lex ideal is not almost lex in general, since it may not be saturated.
Both lex ideals and almost lex ideals are strongly stable.

\begin{examples}\label{ExMonomialIdeals}
Let  $e_1=2, e_2 = 3, e_3=e_4 = \infty$.
The  associated Clements-Lindstr\"om ring is $A = \frac{\Bbbk[x_1, x_2, x_3, x_4]}{(x_1^2, x_2^3)}$.
We consider the following ideals:
\begin{itemize}
\item $I = \big( x_1 x_2,\, x_2^2,x_1x_3^2,x_2x_3^2,x_3^4\big) \in \Hilb^8(\Proj A)$ 
is strongly stable, but it is neither lex nor almost lex, as $x_1x_3 >_{\lex} x_2^2$.
\item 
$J = \big( x_1 x_2,\, x_1x_3,\, x_2^2,\, x_2x_3,\, x_3^6\big)\in \Hilb^8(\Proj A)$  is almost lex, but not lex, as $x_1x_4 >_{\lex} x_2^2$.
\item $K = \big(x_1x_2,\, x_1x_3,\, x_1x_4,\, x_2^2,\, x_2x_3^2,\, x_2x_3x_4^4,\, x_2x_4^6,\, x_3^8\big) = \Lex(J)$ is lex, but not almost lex, as $K:x_4 \ne K$.
Its saturation  $L = \big(x_1,\,x_2,\, x_3^8\big) \in \Hilb^8(\Proj A)$ is lex and almost lex.
\item $C = \big(x_1x_2,\, x_1x_3^2,\, x_2x_3^2,\, x^2x_3,\, x_3^3\big) \in \Hilb^8(\Proj A)$  is almost lex. 
This is an example of the  ideals that will play an important role in Section \ref{SectionSyzygies}.
\end{itemize}
\end{examples}

If $\Hilb^d(\Proj A)\ne \emptyset$,
then there is exactly one lex ideal in  $\Hilb^d(\Proj A)$.
We emphasize that  the lex ideal of a given Hilbert function and the (saturated) lex ideal of a given multiplicity  are different concepts.
The notation $\Lex(I)$ is reserved for the lex ideal with the same Hilbert function as $I$.
We remark that there are algorithms to compute  all strongly stable  or almost lex ideals of $\Hilb^d(\PP^n)$ \cite{AlLe18,CiLeMaRo11,MoNa14},
and these algorithms can be extended to the more general setting of Clements-Lindstr\"om schemes $\Proj A$.

\section{Decomposition of monomial ideals}\label{SectionDecomposition}

We introduce a recursive decomposition of 
  ideals in Clements-Lindstr\"om rings.
This decomposition is particularly effective for strongly stable and almost lex ideals,
and it  will play a fundamental role in our study of syzygies in the subsequent sections.

\begin{notation}\label{NotationRings}
For the rest of the paper, we fix the following rings:
\begin{equation*}
\begin{aligned}
 S & =  \Bbbk[x_1, x_2,\ldots, x_{n}, x_{n+1}],  \qquad &R & =  S / \big(x_1^{e_1}, \ldots, x_n^{e_n}\big),\\  
\ovS &= \Bbbk[x_1, x_2,\ldots, x_{n-1}, x_{n+1}], \qquad & \ovR & =  \ovS / \big(x_1^{e_1}, \ldots, x_{n-1}^{e_{n-1}}\big),\\
\wS &= \Bbbk[x_1, x_2,\ldots, x_{n-1}, x_n], \qquad &\wR &=  {\wS}/{\big(x_1^{e_1}, \ldots, x_n^{e_n}\big)},
\end{aligned}
\end{equation*}
where  $2\leq e_1 \leq e_2 \leq \cdots \leq e_{n}\leq \infty$.
That is, we set  $e_{n+1} = \infty$, 
and $x_{n+1}^{e_{n+1}}=0$  will be omitted. 
We use $\overline{I}$ and $\tilde{I}$ to denote the image of an ideal $I\subseteq S$, respectively, $I \subseteq R$, 
in the factor rings $\ovS$ and $\wS$,  respectively, $\ovR$ and $\wR$.
\end{notation}

The ring $\wR$ is an algebra retract of $R$, since  $R = \wR[x_{n+1}]$,
and thus it may be regarded both as a subring and as a factor ring of $R$; 
both points of view will be  useful in this paper.
This fact, together with the short exact sequence
$
0 \to R/I \xrightarrow{x_{n+1}} R/I \to \wR/\wI \to 0,
$
 induces a tight relation between  ideals of $R$ and $\wR$, 
and we summarize the main formulas in the next remark.

\begin{remark}\label{RemarkModuloXNplus1}
Let $I\subseteq R$ be an ideal  such that $I:x_{n+1}=I$.
For all $d\in \ZZ$
we have  $\HF(\wI, d) = \HF(I,d)-\HF(I,d-1)$.
Moreover, for all $i,j\in \NN$ we have
\begin{align*}
\beta_{i,j}^\wS(\wR/\wI) &= \beta_{i,j}^S(R/I),
&\beta_{i,j}^S(\wR/\wI) &= \beta_{i,j}^S(R/I) + \beta_{i-1,j-1}^S(R/I),\\
\beta_{i,j}^\wR(\wR/\wI) &= \beta_{i,j}^R(R/I),
&\beta_{i,j}^R(\wR/\wI) &= \beta_{i,j}^R(R/I) + \beta_{i-1,j-1}^R(R/I).
\end{align*}
If $I$ is  strongly stable, then so is $\wI$.
Conversely,   the extension $ K R \subseteq R$ of a strongly stable $K\subseteq \wR$ is a saturated strongly stable ideal
whose image in $\wR$ is $K$.
\end{remark}

\begin{prop}\label{PropositionFiniteLengthRank}
Let $I\subseteq J \subseteq R$ be monomial ideals 
such that $I:x_{n+1}=I, J:x_{n+1}=J$, and
$\dim(R/I), \dim(R/J)\leq 1$.
The quotient $J/I$ is a finite free module over $\Bbbk[x_{n+1}]$ via restriction of scalars,
with
$
\rank_{\Bbbk[x_{n+1}]} \left({J}/{I} \right) =\dim_\Bbbk \big({\wJ}/{\wI}\big)= \mult(I)-\mult(J).
$
\end{prop}

\begin{proof}
Denoting $M = J/I$ and $\widetilde{M} = \wJ/\wI$, we have
 $M \cong \widetilde{M} \otimes_\wR \wR[x_{n+1}] \cong \widetilde{M} \otimes_\Bbbk \Bbbk[x_{n+1}]$.
 This implies  the first statement  and 
 $\rank_{\Bbbk[x_{n+1}]} \left({J}/{I} \right) = \dim_\Bbbk \big({\wJ}/{\wI}\big)$.
For the other equality, we have
 $\rank_{\Bbbk[x_{n+1}]}(M)= \HP(M) = \HP(R/I) -\HP(R/J) = \mult(I)-\mult(J)$.
\end{proof}

For a monomial ideal  $I \subseteq R$, 
there exist uniquely determined monomial ideals $I_\ell \subseteq \ovR$ such that
the following decomposition of $\ovR$-modules holds
\begin{equation}\label{EqDecomposition}
I = \bigoplus_{\ell=0}^{e_{n}-1} I_\ell x_{n}^\ell.
\end{equation}
The set of components $\{I_\ell\}$ is finite if $e_n < \infty$, infinite otherwise.
Throughout the paper, the notation $I_\ell$ will always refer to this decomposition; 
it should not be confused with graded components,  denoted instead by $[I]_j$.

In the next proposition, 
we list the basic properties of the decomposition \eqref{EqDecomposition}.

\begin{prop}\label{PropositionElementaryPropertiesDecomposition}
Let $R$ be a Clements-Lindstr\"om ring
and $I \subseteq R $ a monomial ideal such that $I:x_{n+1}=I$ and $\dim(R/I)=1$.
\begin{itemize}
\item[(1)] The sequence $\{ I_\ell\}$ is a non-decreasing chain of ideals of $\ovR$.
\item[(2)] If $e_n = \infty $, 
then $I_\ell = \ovR$ for $\ell \gg 0$.
\item[(3)] $I_\ell$ is saturated with $\dim(\ovR/I_\ell)= 1$ for all $\ell= 0,\ldots, e_n-1$. 
\item[(4)] $\mult(I) = \sum_{\ell=0}^{e_n-1} \mult(I_\ell)$.
\item[(5)] $I$ is strongly stable if and only if $I_\ell$ is strongly stable for all $\ell= 0,\ldots, e_n-1$ and
$(x_1, \ldots, x_{n-1}) I_\ell \subseteq I_{\ell-1}$ for all $\ell= 1,\ldots, e_n-1$. 
\item[(6)]
The quotient
 $I_{\ell}/I_{\ell-1}$ is a  free $\Bbbk[x_{n+1}]$-module with rank  $\mult(I_{\ell-1})-\mult(I_{\ell}) $
 for all $\ell= 1,\ldots, e_n-1$. 
\end{itemize}
\end{prop}

\begin{proof}
Item (1) follows from \eqref{EqDecomposition}, 
since $I$ is closed under multiplication by $x_n$.
Since $\dim(R/I)=1$ and $I : x_{n+1} = I$,
we have $\sqrt{I}= (x_1, \ldots, x_n)$, thus,
if $e_n =\infty$, we have $x_n^\ell \in I$ and $I_\ell = \ovR$ for $\ell \gg 0$, proving (2).
Observe that each monomial generator of $I_\ell$ divides a monomial generator of $I$;
hence, the generators of $I_\ell$ are coprime with $x_{n+1}$, and so each $I_\ell$ is saturated.
We have $\sqrt{I_\ell}=(x_1, \ldots, x_{n-1})$,
 since $\sqrt{I}= (x_1, \ldots, x_n)$, 
and this concludes the proof of (3).
Item (4) follows from \eqref{EqDecomposition}  and item (3),
since $\mult(I)$ is the asymptotic value of $\HF(R/I)$.
Item (5) follows by definition of strongly stable ideals.
Finally, (6) follows directly from Proposition \ref{PropositionFiniteLengthRank}.
\end{proof}

\begin{example}
Let $R = {\Bbbk[x_1,x_2, x_3, x_4]}/{\big(x_1^2,x_2^3\big)}$,
so 
 $\ovR = {\Bbbk[x_1,x_2, x_4]}/{\big(x_1^2,x_2^3\big)}$.
Consider the saturated strongly stable ideal $I = \big( x_1 x_2,\, x_2^2,x_1x_3^2,x_2x_3^2,x_3^4\big)\subseteq R$
of Examples \ref{ExMonomialIdeals}.
The components of $I$ are the $\ovR$-ideals
$$
I_0 = I_1 = \big(x_1x_2,x_2^2\big), \qquad I_2= I_3 = \big( x_1, x_2\big), \qquad I_\ell = \ovR  \quad\text{for } \ell \geq 4.
$$
We have $\mult(I) =  8$,
and the sequence $\{\mult(I_\ell)\}$ is $\{3,3,1,1,0,0,\ldots\}$.

\end{example}

\section{Maximal syzygies}\label{SectionSyzygies}

We begin this section by studying a special almost lex ideal in $R$, which plays a
 central role in the extremality of syzygies in $\Hilb^d(\Proj R)$.

\begin{definition}\label{DefCd}
Let $R$ be a Clements-Lindstr\"om ring and  $d \in \NN$
with $\Hilb^d(\Proj R) \ne 0$.
We let  $\CC(d,R)$ or $\CC(d)$ denote the  unique almost lex ideal $C \in \Hilb^d(\Proj R)$ such that 
$$(x_1,\ldots,x_n)^{p+1}\subseteq C \subseteq (x_1,\ldots,x_n)^p \qquad \text{for some }p \in \NN.
$$
If $I\in  \Hilb^d(\Proj R)$, we define $\CC(I) := \CC(d, R)$.
\end{definition}

The ideal $\CC(d,R)$ is generated by $(x_1, \ldots, x_n)^{p+1}$
and by an initial $<_\lex$-segment of the vector space $[(x_1, \ldots, x_n)^p]_p$,
such that $\mult(\CC(d,R))=d$.
It is clear that such $\CC(d,R)$ is unique for every $d\in \NN$, and exists as long as $\Hilb^d(\Proj R) \ne 0$.
In Examples \ref{ExMonomialIdeals}, we have $C=\CC(8,R)$.

The next proposition highlights the key extremal features of the ideal $\CC(d,R)$.

\begin{prop}\label{PropositionMainPropertiesCd}
Let $R$ be a Clements-Lindstr\"om ring and $d \in \NN$.
\begin{enumerate}
\item[(C1)\label{C1}] $ \CC(d,R)$ is almost lex.

\item[(C2)\label{C2}]  Every component $ \CC(d,R)_\ell \subseteq \ovR$ is equal to $\CC(d_\ell,\ovR)$ for some   $d_\ell\in\NN$.

\item[(C3)\label{C3}] If $d_1 < d_2 $, then   $\CC(d_2,R) \subseteq \CC(d_1,R)$.

\item[(C4)\label{C4}]  If  $I\in \Hilb^d(\Proj R)$  is strongly stable,  then,  for every $\rho < e_n$, we have
$$
\sum_{\ell = 0}^\rho \mult\big(I_\ell\big) \leq  \sum_{\ell = 0}^\rho  \mult\big(\CC(d,R)_\ell\big).
$$
\item[(C5)\label{C5}]  If  $I\in \Hilb^d(\Proj R)$  is   strongly stable, then  
$$\mult\big((x_1, \ldots, x_n)I\big) \leq
\mult\big((x_1, \ldots, x_n)\CC(d,R)\big).$$
\end{enumerate}
\end{prop}
\begin{proof}
We prove the proposition by induction on $n$.
The case  $n=0$ is trivial, since $(0)$ and $R$ are the only saturated ideals of $R$, so we
assume $n>0$.
Properties \hyperref[C1]{(C1)},
 \hyperref[C2]{(C2)} and \hyperref[C3]{(C3)} follow immediately by Definition \ref{DefCd}.

We  prove \hyperref[C4]{(C4)} by induction on $d$. 
The case $d=0$ is trivial, so let $d>0$.
Assume by contradiction there is a strongly stable  $I\in \Hilb^d(\Proj R)$ violating  \hyperref[C4]{(C4)}.
Define 
$$
J = \bigoplus_{\ell=0}^{e_n-1}\CC(I_\ell)x_n^\ell \subseteq R.
$$
We claim that $J$ is a saturated strongly stable ideal of $R$.
We have $I_\ell \subseteq I_{\ell+1}$ and $\mult(I_\ell) \geq \mult(I_{\ell+1})$ for every $\ell$,
hence $\CC(I_\ell) \subseteq \CC(I_{\ell+1})$ by \hyperref[C3]{(C3)}, and this implies that  $J$ is an ideal of $R$.
Moreover, $J$ is saturated, since $J:x_{n+1} = J$.
To show that $J$ is strongly stable, we use Proposition \ref{PropositionElementaryPropertiesDecomposition} (5).
Each component $J_\ell = \CC(I_\ell)$ is almost lex by \hyperref[C1]{(C1)}, and, in particular,
strongly stable.
It remains to show that $(x_1, \ldots, x_{n-1}) \CC(I_\ell) \subseteq \CC(I_{\ell-1})$
for $1 \leq \ell < e_n$.
It follows by Definition \ref{DefCd} that $ (x_1, \ldots, x_{n-1}) \CC(I_\ell) =\CC(d',\ovR)$ for some $d' \in \mathbb{N}$,
thus, by \hyperref[C3]{(C3)},
it suffices to show that $\mult((x_1, \ldots, x_{n-1}) \CC(I_\ell)) \geq \mult(\CC(I_{\ell-1}))$.
We have $(x_1, \ldots, x_{n-1})I_\ell \subseteq I_{\ell-1}$,
 since $I$ is strongly stable.
Combining with \hyperref[C5]{(C5)}, we get
$$
 \mult\big((x_1, \ldots, x_{n-1}) \CC(I_\ell)\big) \geq \mult\big((x_1, \ldots, x_{n-1}) I_\ell\big) \geq \mult(I_{\ell-1}) =\mult\big(\CC(I_{\ell-1})\big),
$$
yielding the desired inequality and completing the proof of the claim.

To summarize, there exists a strongly stable $J\in \Hilb^d(\Proj R)$ violating  \hyperref[C4]{(C4)} and such that $J_\ell = \CC(J_\ell)\subseteq \ovR$ for every $\ell$.
Let $\ell_1$ denote the least $\rho$ for which \hyperref[C4]{(C4)} fails for $J$.
Then,
 $\mult(J_\ell) = \mult(\CC(d,R)_\ell)$ for $\ell<\ell_1 $
and 
$\mult(J_{\ell_1}) > \mult(\CC(d,R)_{\ell_1})$.
Since $\mult(J) = \mult(\CC(d,R))$, by Proposition \ref{PropositionElementaryPropertiesDecomposition}
(4) there is some $\ell_2 > \ell_1$ such that 
$\mult(J_{\ell_2}) < \mult(\CC(d,R)_{\ell_2})$.
By \hyperref[C3]{(C3)}, 
it  follows that 
$J_{\ell_1} \subsetneq \CC(d,R)_{\ell_1}$
and 
$J_{\ell_2}\supsetneq\CC(d,R)_{\ell_2}$.

Let $p \in \NN$ be the integer such that $(x_1,\ldots,x_n)^{p+1}\subseteq \CC(d,R) \subsetneq (x_1,\ldots,x_n)^p$,
then $(x_1,\ldots,x_{n-1})^{p+1-\ell}\subseteq \CC(d,R)_\ell \subseteq (x_1,\ldots,x_{n-1})^{p-\ell}$ for every $\ell$.
We have 
$$(x_1, \ldots, x_{n-1})^{p+1-\ell_1} \subseteq J_{\ell_1} \subsetneq \CC(d,R)_{\ell_1}\subseteq (x_1, \ldots, x_{n-1})^{p-\ell_1}, $$
where the first inclusion holds since 
since $(x_1, \ldots, x_{n-1})^{p+1-\ell_2} \subseteq \CC(d,R)_{\ell_2} \subseteq J_{\ell_2}$,
and $(x_1, \ldots, x_{n-1})^{\ell_2-\ell_1} J_{\ell_2} \subseteq J_{\ell_1}$  by Proposition \ref{PropositionElementaryPropertiesDecomposition} (5).
Thus,  there is a monomial  generator $\bfu$ of $\CC(d,R)_{\ell_1}$ such that $\bfu \notin J_{\ell_1}$ and $\deg(\bfu) = p-\ell_1$.
Furthermore, 
we have 
$$(x_1, \ldots, x_{n-1})^{p+1-\ell_2} \subseteq 
\CC(d,R)_{\ell_2}
\subsetneq J_{\ell_2} \subseteq (x_1, \ldots, x_{n-1})^{p-\ell_2} $$
since 
if $J_{\ell_2} \not\subseteq (x_1, \ldots, x_{n-1})^{p-\ell_2} $
then  
$(x_1, \ldots, x_{n-1})^{\ell_2-\ell_1}J_{\ell_2} \not\subseteq (x_1, \ldots, x_{n-1})^{p-\ell_1} $,
so, by Proposition \ref{PropositionElementaryPropertiesDecomposition} (5), $J_{\ell_1} \not\subseteq (x_1, \ldots, x_{n-1})^{p-\ell_1} $, contradiction.
Thus, there is a  a monomial generator $\bfv$ of $J_{\ell_2}$ such that $\bfv \notin\CC(d,R)_{\ell_2}$ and
 $\deg(\bfv)= p-\ell_2$.
We have $\bfu x_n^{\ell_1}\in \CC(d,R)$, $\bfv x_n^{\ell_2}\notin \CC(d,R)$,
and both monomials have degree $p$, so necessarily 
$\bfu x_n^{\ell_1}>_\lex\bfv x_n^{\ell_2}$.
This implies 
$\bfu x_n^{\ell_1}\geq _\lex\bfv x_{n-1}^{\ell_2-\ell_1}x_n^{\ell_1}$ 
and, hence,
$\bfu \geq _\lex\bfv x_{n-1}^{\ell_2-\ell_1}$.
Since $\bfu \notin J_{\ell_1}$ and $J_{\ell_1}$ is almost lex, 
we see that $\bfv x_{n-1}^{\ell_2-\ell_1}\notin J_{\ell_1}$.
This is a contradiction, since $\bfv \in J_{\ell_2}$ and 
$(x_1, \ldots, x_{n-1})^{\ell_2-\ell_1}J_{\ell_2}\subseteq J_{\ell_1}$.
The proof of \hyperref[C4]{(C4)} is concluded.

In order to prove \hyperref[C5]{(C5)},
we begin by observing that 
 $(x_1, \ldots, x_{n}) I $ has the decomposition
$$
(x_1, \ldots, x_{n}) I = (x_1, \ldots, x_{n-1}) I_0 \oplus \bigoplus_{\ell = 1}^{e_n-1} 
\big(I_{\ell-1} +  (x_1, \ldots, x_{n-1}) I_\ell \big)
x_{n}^{\ell}.
$$
However, we have 
$ (x_1, \ldots, x_{n-1}) I_\ell \subseteq I_{\ell-1} $
by Proposition \ref{PropositionElementaryPropertiesDecomposition} (5),
so we may rewrite
$$
(x_1, \ldots, x_{n}) I = (x_1, \ldots, x_{n-1}) I_0 \oplus \bigoplus_{\ell = 1}^{e_n-1} 
I_{\ell-1} x_{n}^{\ell}
$$
and, by Proposition \ref{PropositionElementaryPropertiesDecomposition} (4),  this implies that 
\begin{equation}
\label{EqC51}
\mult\big( (x_1, \ldots, x_{n}) I  \big) = \mult\big( (x_1, \ldots, x_{n-1}) I_0  \big) + 
\sum_{\ell = 0}^{e_n-2} \mult(I_{\ell}).
\end{equation}
By the same argument, we have
\begin{equation}
\label{EqC52}
\mult\big( (x_1, \ldots, x_{n}) \CC(d,R)  \big) = \mult\big( (x_1, \ldots, x_{n-1}) \CC(d,R)_0  \big) + 
\sum_{\ell = 0}^{e_n-2} \mult(\CC(d,R)_{\ell}).
\end{equation}
Using \hyperref[C4]{(C4)} with $\rho = 0$, we get $\mult(I_0) \leq \mult(\CC(d,R)_0)$, 
hence,
 $\CC(d,R)_0 \subseteq \CC(I_0)$ by \hyperref[C3]{(C3)}.
It follows that 
$(x_1, \ldots, x_{n-1})\CC(d,R)_0 \subseteq (x_1, \ldots, x_{n-1})\CC(I_0)$
and
\begin{equation*}
\mult\big((x_1, \ldots, x_{n-1})\CC(d,R)_0\big) \geq \mult\big( (x_1, \ldots, x_{n-1})\CC(I_0)\big)
\geq 
\mult\big( (x_1, \ldots, x_{n-1})I_0\big),
\end{equation*}
where the last inequality follows by applying \hyperref[C5]{(C5)} to $I_0$.
On the other hand, 
using \hyperref[C4]{(C4)} with $\rho = e_n-2$,
 we also see that 
$\sum_{\ell = 0}^{e_n-2} \mult(I_{\ell})
\leq
\sum_{\ell = 0}^{e_n-2} \mult(\CC(d,R)_{\ell})$.
Comparing \eqref{EqC51} and \eqref{EqC52},
we have proved \hyperref[C5]{(C5)}.
\end{proof}

\begin{remark}
Proposition \ref{PropositionMainPropertiesCd} captures the essential properties needed to obtain  sharp upper  bounds for the syzygies.
Moreover, the assignment 
$
d \mapsto \CC(d,R)\in\Hilb^d(\Proj R)
$
is uniquely characterized by the  properties of Proposition \ref{PropositionMainPropertiesCd},
as it  follows  by induction on $n$ using  \hyperref[C2]{(C2)} and \hyperref[C4]{(C4)}.
 One may thus give a recursive construction of  $\CC(d,R)$
based on these axioms.
This less explicit but effective approach
might be the basis for extending the methods and results of this paper to other classes of rings $R$ or other Hilbert schemes.
\end{remark}

The most important property  of $\CC(d,R)$ is  \hyperref[C4]{(C4)}.
As the following equivalent formulation shows, 
it is closely related to  similar inequalities about ``hypersurface sections'',  see for instance
 \cite[Lemma 3.3]{CaSa18},
\cite[Theorem 2.2]{Ga99}, or the main theorem in 
\cite{HePo98}.

\begin{cor}\label{CorollaryHyperplaneSections}
 Let  $J\in \Hilb^d(\Proj R)$  be  strongly stable.
 For every $0 \leq h < e_n$, we have
$
\mult\big(J + (x_n^h)\big) \leq \mult\big(\CC(d) + (x_n^h)\big).
$
\end{cor}

We now turn our focus to  the study of syzygies of ideals $I \in \Hilb^d(\Proj R)$.
The  results of \cite{MeMu11} allow us to perform an important  reduction to strongly stable ideals.

\begin{lemma}\label{LemmaAlmostLexMerminMurai}
For every $I\in\Hilb^d(\Proj R)$ there exists a strongly stable
 $J\in\Hilb^d(\Proj R)$ with $\beta_{i}^S(R/I) \leq \beta_{i}^S(R/J)$ for all $i\geq 0$.
\end{lemma}
\begin{proof}
Since $I\subseteq R$ is saturated, there exists $z\in [R]_1$ that is a non-zerodivisor on $R/I$.
Up to a change of coordinates, we may assume  $z=x_{n+1}$.
By \cite[Proposition 8.7]{MeMu11},
there exists a strongly stable ideal 
$K\subseteq \wR$ with $\HF(K) = \HF(\tilde{I})$
and $\beta_{i,j}^{\tilde{S}}(\tilde{R}/\tilde{I}) \leq \beta_{i,j}^{\tilde{S}}(\tilde{R}/K)$ for all $i,j$.
The conclusion follows from Remark \ref{RemarkModuloXNplus1} considering the extension $J = KR$.
\end{proof}

In the next lemma,
 we consider the natural $\ZZ^{n+1}$-grading on $R$.

\begin{lemma}\label{LemmaUpperBoundBettiFreeModule}
Let $M$ be a finite $\ZZ^{n+1}$-graded $R$-module that is a finite free $\Bbbk[x_{n+1}]$-module of  rank $r$ via restriction of scalars.
For every $i\in \NN$, we have
\begin{itemize}
\item[(i)]
 $\beta_i^S(M) \leq r\cdot \beta_i^S(\Bbbk[x_{n+1}])$  and $\beta_i^R(M) \leq r\cdot \beta_i^R(\Bbbk[x_{n+1}])$;
\item[(ii)]
$\beta_i^S(M) = r\cdot \beta_i^S(\Bbbk[x_{n+1}])$
and
$\beta_i^R(M) = r\cdot \beta_i^R(\Bbbk[x_{n+1}])$
if $\ann_R(M) = (x_1, \ldots, x_n)$.
\end{itemize}
\end{lemma}
\begin{proof}
We prove (ii) first.
Let $m_1, \ldots, m_s$ be minimal $\ZZ^{n+1}$-graded $R$-module generators of $M$.
The assumptions imply the isomorphisms of $R$-modules
$M \cong Rm_1 \oplus \cdots\oplus Rm_s $
and $Rm_h \cong \Bbbk[x_{n+1}]$ for every $h$,
therefore, 
$s = r$ and the formulas for the Betti numbers follow.
To prove (i), we may assume $r>1$.
Let $M' = (x_1, \ldots, x_n) M$ and $M'' = M/M'$.
Both $M'$ and $M''$ are finite $\ZZ^{n+1}$-graded $R$-modules.
As $\Bbbk[x_{n+1}]$-modules via restriction of scalars,
 $M'$ is free  of rank less than $r$,
whereas $M''$ is also free, by multidegree reasons, and it satisfies (ii).
The  conclusions  follow, by induction on $r$, from the  exact sequence $0\rightarrow M' \rightarrow M \rightarrow M'' \rightarrow 0$.
\end{proof}

We are ready to present the main result.

\begin{thm}\label{TheoremExtremalFiniteResolution}
Let $S = \Bbbk[x_1, \ldots, x_{n+1}]$ be a polynomial ring  and $R= S/(x_1^{e_1}, \ldots, x_n^{e_n})$  a Clements-Lindstr\"om ring, where $2 \leq e_1 \leq \cdots \leq e_n \leq \infty$.
For each $d\in \NN$, we have
$$
\beta_{i}^S\big(R/I\big) \leq \beta_{i}^S\big(R/\CC(d)\big) 
$$
for all $I\in\Hilb^d(\Proj R)$ and  all $i \geq 0$.
\end{thm}

\begin{proof}
We proceed by induction on $n$. 
The case $n=0$ is trivial, so  let $n>0$.
By Lemma \ref{LemmaAlmostLexMerminMurai}, we may assume without loss of generality that $I$ is strongly stable.
Let $\mI, \mC$ denote the preimages of  $I, \CC(d) \subseteq R$ in the polynomial ring $S$.
There are decompositions 
\begin{equation}\label{EqDecompositionIandE}
\mI = \bigoplus_{\ell=0}^\infty \mI_\ell x_n^\ell
\quad
\text{and}
\quad
\mC = \bigoplus_{\ell=0}^\infty \mC_\ell x_n^\ell,
\end{equation}
where $\mI_\ell, \mC_\ell$ are ideals of $\ovS$. 
Specifically, $\mI_\ell \subseteq \ovS$ is the preimage of $I_\ell \subseteq \ovR$ if $\ell < e_n$, and $\mI_\ell = \ovS $ if $e_n \leq \ell < \infty$; 
likewise for $\mC_\ell$.
Since $R/I \cong S/\mI$ and   $R/\CC(d) \cong S/\mC$,
we must prove that $\beta_{i}^S(\mI) \leq \beta_{i}^S(\mC)$ for all $i$.
The variable $x_n$ is a non-zerodivisor on $S, \mI, \mC,$ so it suffices to 
prove $\beta_{i}^\ovS(\mI/x_n\mI) \leq \beta_{i}^\ovS(\mC/x_n\mC)$ for all $i$.
  
Let $\mJ\subseteq \ovS$ be the preimage of  $\CC(I_0) \subseteq \ovR$.
Since $\ovS/\mI_0 \cong R/I_0$ and $\ovS/\mJ \cong \ovR / \CC(I_0)$, 
we have $  \beta_i^\ovS(\mI_0) \leq \beta_i^\ovS(\mJ)$ for every $i \geq 0$  by induction.
Applying  \hyperref[C4]{(C4)} with $\rho =0$,
we get $\mult(I_0 )\leq\mult(\CC(d)_0)$,
and from \hyperref[C3]{(C3)} we deduce  $\CC(d)_0 \subseteq \CC(I_0)$ and, 
hence,
$\mC_0 \subseteq \mJ$.
By Proposition \ref{PropositionFiniteLengthRank},
the quotient $\mJ/\mC_0 \cong \CC(I_0)/ \CC(d)_0$ is a  free $\Bbbk[x_{n+1}]$-module   of rank 
 $r_0 = \mult(\CC(d)_0)- \mult(\CC(I_0))$. 
Applying Lemma \ref{LemmaUpperBoundBettiFreeModule} (i) to  the short exact sequence $ 0\rightarrow \mC_0 \rightarrow \mJ \rightarrow \mJ/\mC_0 \rightarrow 0$,
we obtain
\begin{equation}\label{EqFirstEstimateFiniteResolution}
 \beta_i^\ovS(\mI_0) \leq 
\beta_i^\ovS ( \mJ) \leq \beta_i^\ovS ( \mC_0) +  \beta_i^\ovS ( \mJ/\mC_0)  \leq  \beta_i^\ovS ( \mC_0) +  r_0\beta_i^\ovS\big(\Bbbk[x_{n+1}]\big).
\end{equation}

First, assume that $e_n = \infty$. 
From \eqref{EqDecompositionIandE},
we deduce
 decompositions of $\ovS$-modules 
 \begin{equation}\label{EqDecompositionQuotientsIandE}
 \quad
\frac{\mI}{x_n\mI}\cong \mI_0 \oplus \bigoplus_{\ell=1}^\infty \frac{\mI_\ell}{\mI_{\ell-1}} \cong \mI_0 \oplus \bigoplus_{\ell=1}^\infty \frac{I_\ell}{I_{\ell-1}},
 \quad
 \frac{\mC}{x_n\mC}\cong \mC_0\oplus \bigoplus_{\ell=1}^\infty \frac{\mC_\ell}{\mC_{\ell-1}} \cong \mC_0\oplus \bigoplus_{\ell=1}^\infty \frac{\CC(d)_\ell}{\CC(d)_{\ell-1}}. 
 \end{equation}
Applying  Proposition \ref{PropositionElementaryPropertiesDecomposition} (2) and (6) we see that
the terms $\bigoplus_{\ell=1}^\infty \frac{I_\ell}{I_{\ell-1}}$ and  $\bigoplus_{\ell=1}^\infty \frac{\CC(d)_\ell}{\CC(d)_{\ell-1}}$ are   free $\Bbbk[x_{n+1}]$-modules of ranks $ r_1 = \mult(I_0)$ and $r_2 = \mult(\CC(d)_0)$, respectively.
Moreover, by Proposition \ref{PropositionElementaryPropertiesDecomposition} (5),  
they are annihilated by $(x_1, \ldots, x_{n-1})$.
Using Lemma \ref{LemmaUpperBoundBettiFreeModule} (ii) and 
combining with \eqref{EqFirstEstimateFiniteResolution}, 
we obtain
\begin{align*}
\beta_{i}^\ovS(\mI/x_n\mI) & = \beta_{i}^\ovS(\mI_0) + \beta_{i}^\ovS\left( \bigoplus_{\ell=1}^\infty \frac{\mI_\ell}{\mI_{\ell-1}}\right) =  \beta_{i}^\ovS(\mI_0) + r_1\beta_{i}^\ovS\big(\Bbbk[x_{n+1}]\big)\\
& \leq   \beta_i^\ovS ( \mC_0)  +  (r_0+ r_1)\beta_{i}^\ovS\big(\Bbbk[x_{n+1}]\big).
\end{align*}
Finally, we have  $ \beta_i^\ovS ( \mC/x_n\mC) =    \beta_i^\ovS ( \mC_0)  +  (r_0+ r_1)\beta_{i}^\ovS\big(\Bbbk[x_{n+1}]\big)$ by \eqref{EqDecompositionQuotientsIandE} and Lemma \ref{LemmaUpperBoundBettiFreeModule} (ii), 
since $\bigoplus_{\ell=1}^\infty \frac{\CC(d)_\ell}{\CC(d)_{\ell-1}}$ has rank $r_2 = r_0+r_1$.
This concludes the proof in this case. 

Now, assume  $e_n < \infty$.
The decompositions of $\ovS$-modules obtained from  \eqref{EqDecompositionIandE} become 
 \begin{equation}\label{EqSecondDecompositionQuotientsIandE}
\frac{\mI}{x_n\mI}\cong\mI_0 \oplus \bigoplus_{\ell=1}^{e_n-1} \frac{I_\ell}{I_{\ell-1}} \oplus \frac{\ovR}{I_{e_n-1}}
 \quad
 \text{and}
 \quad
 \frac{\mC}{x_n\mC}\cong \mC_0\oplus \bigoplus_{\ell=1}^{e_n-1} \frac{\CC(d)_\ell}{\CC(d)_{\ell-1}} \oplus \frac{\ovR}{\CC(d)_{e_n-1}}.
 \end{equation}
Our goal is to estimate $\beta_{i}^\ovS(\ovR/I_{e_n-1})$.
By induction, we have 
$\beta_{i}^\ovS(\ovR/I_{e_n-1}) \leq \beta_{i}^\ovS\big(\ovR/\CC(I_{e_n-1})\big)
$ 
for all $i\geq 0$.
Using 
 Proposition \ref{PropositionElementaryPropertiesDecomposition} (4)
and 
\hyperref[C4]{(C4)} 
with $\rho = e_n-2$
we see that
\begin{align*}
\sum_{\ell=0}^{e_n-1} \mult(I_\ell) &= \mult(I) =  d
=\mult(\CC(d)) = \sum_{\ell=0}^{e_n-1} \mult(\CC(d)_\ell)
\intertext{and}
  \sum_{\ell=0}^{e_n-2} \mult(I_\ell) &\leq \sum_{\ell=0}^{e_n-2} \mult(\CC(d)_\ell),
\end{align*}
implying that $\mult(I_{e_n-1}) \geq \mult(\CC(d)_{e_n-1})$, and, thus,
$\CC(I_{e_n-1}) \subseteq \CC(d)_{e_n-1}$,
by \hyperref[C3]{(C3)}.
The  exact sequence 
$
0 \rightarrow {\CC(d)_{e_n-1}}/{\CC(I_{e_n-1})} \rightarrow {\ovR}/{\CC(I_{e_n-1})} \rightarrow {\ovR}/{\CC(d)_{e_n-1}} \rightarrow  0
$
yields
\begin{equation}\label{EqEstimateLastPart}
\beta_{i}^\ovS(\ovR/I_{e_n-1}) \leq \beta_{i}^\ovS\big(\ovR/\CC(I_{e_n-1})\big) \leq  \beta_{i}^\ovS\left(\frac{\CC(d)_{e_n-1}}{\CC(I_{e_n-1})}\right) + \beta_{i}^\ovS\left(\frac{\ovR}{\CC(d)_{e_n-1}}\right).
\end{equation}

Finally, we are going to  use \eqref{EqSecondDecompositionQuotientsIandE} to give an upper bound for $\beta_i^\ovS(\mI/x_n\mI)$.
As before, 
the $\ovS$-modules 
$\bigoplus_{\ell=1}^{e_n-1} \frac{I_\ell}{I_{\ell-1}}$ and  $\bigoplus_{\ell=1}^{e_n-1} \frac{\CC(d)_\ell}{\CC(d)_{\ell-1}}$ are annihilated by $(x_1, \ldots, x_{n-1})$,
and, 
by  Proposition \ref{PropositionElementaryPropertiesDecomposition} (6),   
they are  free $\Bbbk[x_{n+1}]$-modules of  ranks 
 $ r'_1= \mult(I_0)-\mult(I_{e_n-1})$ and $r'_2 =\mult(\CC(d)_0) - \mult(\CC(d)_{e_n-1})$, respectively.
By Proposition \ref{PropositionFiniteLengthRank},
the module  $\frac{\CC(d)_{e_n-1}}{\CC(I_{e_n-1})}$ is also free over $\Bbbk[x_{n+1}]$, of  rank $r_3 = \mult(\CC(I_{e_n-1}))- \mult(\CC(d)_{e_n-1})$.
Combining the decomposition \eqref{EqSecondDecompositionQuotientsIandE} and the bounds \eqref{EqFirstEstimateFiniteResolution}, \eqref{EqEstimateLastPart}, and using Lemma \ref{LemmaUpperBoundBettiFreeModule} (i), 
we find
\begin{align*}
&\quad\beta_{i}^\ovS(\mI/x_n\mI)  = \beta_{i}^\ovS(\mI_0) + \beta_{i}^\ovS\left( \bigoplus_{\ell=0}^{e_n-1} \frac{I_\ell}{I_{\ell-1}}\right) + \beta_{i}^\ovS\left(\frac{\ovR}{I_{e_n-1}}\right) \\
\leq  &
\left[  \beta_i^\ovS ( \mC_0) +  r_0\beta_i^\ovS\big(\Bbbk[x_{n+1}]\big)\right] 
+ r'_1\beta_i^\ovS\big(\Bbbk[x_{n+1}]\big) + \left[ r_3\beta_i^\ovS\big(\Bbbk[x_{n+1}]\big) +  \beta_{i}^\ovS\left(\frac{\ovR}{\CC(d)_{e_n-1}}\right)\right]\\
= \,&\, \beta_i^\ovS ( \mC_0) +  (r_0+r'_1+r_3)\beta_i^\ovS\big(\Bbbk[x_{n+1}]\big) +  \beta_{i}^\ovS\left(\frac{\ovR}{\CC(d)_{e_n-1}}\right).
\end{align*}
The expression in the last line is equal to 
$ \beta_i^\ovS ( \mC/x_n\mC)$, 
because of \eqref{EqSecondDecompositionQuotientsIandE},  Lemma \ref{LemmaUpperBoundBettiFreeModule} (ii),
and the fact that $r'_2 = r_0+r'_1+r_3$.
This concludes the proof.
\end{proof}

\begin{remark}
The numerical bounds on the Betti numbers provided by Theorem \ref{TheoremExtremalFiniteResolution} can be determined by means of the combinatorial formula in \cite[Proposition 2.1]{Mu08}.
The formula also implies that the bounds are independent of the characteristic of  $\Bbbk$.
\end{remark}

\section{Infinite free resolutions}\label{SectionInfinite}

In this section, 
we investigate  bounds for the  Betti numbers of  the  infinite free resolutions associated to a finite subscheme of a Clements-Lindstr\"om scheme.

We begin by proposing the following  natural problem.

\begin{conj}\label{ConjectureInfiniteResolution}
Let $R$ be a Clements-Lindstr\"om ring. 
We have
$\beta_{i}^{R}(I) \leq \beta_{i}^{R}(\CC(d)) $ 
for every $I\in \Hilb^d(\Proj R)$ and  every $i \geq 0$.
\end{conj}

When the  field has characteristic zero,
the results of \cite{MuPe12}
 reduce the problem to strongly stable ideals.

\begin{lemma}\label{LemmaReductionStableInfiniteResolution}
Assume that  $\ch(\Bbbk)=0$.
For every $I\in \Hilb^d(\Proj R)$ there exists a strongly stable  $J\in \Hilb^d(\Proj R)$
such that $\beta_{i}^R(I)\leq \beta_{i}^R(J)$ for all $i\geq 0$.
\end{lemma}
\begin{proof}
This follows from \cite[Theorem 1.4]{MuPe12},  proceeding exactly as in  Lemma \ref{LemmaAlmostLexMerminMurai}.
\end{proof}

The following theorem is the main result of this section.
The proof employs  a construction from \cite{ArAvHe00,EiPoYu03,GaHiPe02}.

\begin{thm}\label{TheoremIfninite}
Let $S = \Bbbk[x_1, \ldots, x_{n+1}]$ be a polynomial ring  and $R= S/(x_1^{e_1}, \ldots, x_n^{e_n})$  a Clements-Lindstr\"om ring, where $e_j\in \{ 2, \infty\}$ for every $j$.
Assume that $\ch(\Bbbk)=0$.
We have
$
\beta_{i}^R\big(I\big) \leq \beta_{i}^R\big(\CC(d)\big) 
$
for every $I\in\Hilb^d(\Proj R)$ and  every $i \geq 0$.
\end{thm}

\begin{proof}
We proceed by induction on $n$, and the case $n=0$ is trivial, so let $n>0$.
By Lemma \ref{LemmaReductionStableInfiniteResolution}, we may assume that $I$ is strongly stable.
In  addition to Notation \ref{NotationRings},
in this proof we consider the ``intermediate'' ring 
$$
T = \frac{S}{(x_1^{e_1}, \ldots, x_{n-1}^{e_{n-1}})},
$$
so that $R=T/(x_n^{e_n})$.
By assumption,  either $e_n = \infty$, in which case $T=R$,
or $e_n =2$.
Consider the ideal $\mI\subseteq T$ generated by the monomials of $T$ corresponding to the minimal generators of $I\subseteq R$,
that is,
the ideal
$$
\mI = \bigoplus_{\ell=0}^{e_n-1} I_\ell x_n^\ell \oplus \bigoplus_{\ell=e_n}^{\infty} I_{e_n-1} x_n^\ell.
$$
Notice that $\mI$ may be smaller than the preimage of $I$ in $T$ if $e_n=2$, whereas  $\mI=I$  if $e_n=\infty$.
Since $x_{n}$ is a non-zerodivisor on $T$ and $\mI$, and  $T/(x_{n}) \cong \ovR$, 
we have $\beta_{i,j}^{T}(\mI) = \beta_{i,j}^{\ovR}(\mI/x_{n}\mI)$.
We have a decomposition of $\ovR$-modules
\begin{equation}\label{EquationDecompositionRT}
\frac{\mI}{x_{n}\mI} = I_0 \oplus  \bigoplus_{\ell =1}^{e_n-1} \frac{I_\ell}{I_{\ell-1}}.
\end{equation}
By induction, 
$\beta_i^\ovR(I_0)\leq\beta_i^\ovR(\CC(I_0))$.
In the proof of Theorem \ref{TheoremExtremalFiniteResolution},
we established that $\CC(d)_0 \subseteq \CC(I_0)$, 
and that 
$\frac{\CC(I_0)}{\CC(d)_0}$
is a  free $\Bbbk[x_{n+1}]$-module of rank $r_0 = \mult(\CC(d)_0)-\mult(I_0)$.
By Lemma \ref{LemmaUpperBoundBettiFreeModule} (i), we obtain
\begin{equation}\label{InequalityBetti0InfiniteResolution}
\beta_i^\ovR(I_0)\leq\beta_i^\ovR(\CC(I_0)) \leq \beta_i^\ovR(\CC(d)_0) + r_0 \beta_i^{\ovR}(\Bbbk[x_{n+1}]).
\end{equation}
First, assume that 
 $ e_n = \infty$.  
We have seen,  in the proof of Theorem \ref{TheoremExtremalFiniteResolution},
that the $\ovR$-module  
$\oplus_{\ell =1}^{e_n-1} \frac{I_\ell}{I_{\ell-1}}= \oplus_{\ell =1}^{\infty} \frac{I_\ell}{I_{\ell-1}}$ 
is annihilated by $(x_1, \ldots, x_{n-1})$, and is a free $\Bbbk[x_{n+1}]$-module of rank $r_1 =\mult(I_0) $.
By Lemma \ref{LemmaUpperBoundBettiFreeModule}~(ii), 
we get
$
\beta_i^{R}(I) = \beta_i^{\ovR}(I_0) + r_1 \beta_i^{\ovR}(\Bbbk[x_{n+1}]),
$
and, likewise,
$
\beta_i^{R}(\CC(d)) =  \beta_i^{\ovR}(\CC(d)_0) +  r_2 \beta_i^{\ovR}(\Bbbk[x_{n+1}]),
$
where $r_2 = \mult(\CC(d)_0)=r_0+r_1$.
Combining with \eqref{InequalityBetti0InfiniteResolution}, we conclude that $\beta_i^{R}(I) \leq 
\beta_i^{R}(\CC(d))$ as desired.

For the rest of the proof, assume  $e_n = 2$.
The $\ovR$-module  $\oplus_{\ell =1}^{e_n-1} \frac{I_\ell}{I_{\ell-1}} =\frac{I_1}{I_{0}}  $ 
is annihilated by $(x_1, \ldots, x_{n-1})$, and is a free $\Bbbk[x_{n+1}]$-module of rank $r'_1 =\mult(I_0)-\mult(I_{1}) $.
By Lemma \ref{LemmaUpperBoundBettiFreeModule}~(ii) and \eqref{EquationDecompositionRT},
we obtain
\begin{equation}\label{EqBettiNumbersTandOvR}
\beta_i^{T}(\mI) = \beta_i^{\ovR}(I_0) + r'_1\beta_i^{\ovR}(\Bbbk[x_{n+1}]).
\end{equation}

We regard $R,\ovR,$ and $T$ as  $\mathbb{Z}^{n+1}$-graded,
but we  also consider the $\mathbb{Z}$-grading induced by the variable $x_n$ only.
If $M$ is a $\ZZ^{n+1}$-graded $T$-module, 
we define $\sigma(M)$ to be the vector space consisting of the graded components of $M$ with $x_n$-degrees $0$ or $1$.
Clearly, $\sigma$ defines an exact functor from the category of $\ZZ^{n+1}$-graded $T$-modules
to the category of $\ZZ^{n+1}$-graded $\Bbbk$-vector spaces.

Let $\FF$ be the minimal $\mathbb{Z}^{n+1}$-graded free resolution of $\mI$ over $T$.
The $x_{n}$-twists in this  resolution are all equal to 0 or 1:
this follows from the fact that $\FF \otimes_T \frac{T}{(x_{n})}$ is a minimal $\mathbb{Z}^{n+1}$-graded free resolution of $\mI/x_{n}\mI$ over $\ovR$,
and that $\mI/x_{n}\mI$  is generated in $x_n$-degrees  $0,1$.
The complex $\EE=\sigma(\FF)$ is acyclic and  minimal,  in the sense that the image of  its differential lies in $(x_1, \ldots, x_{n+1})\EE$.
Each direct summand in  $\FF$  has the form $T(-\delta_1, \ldots, -\delta_n, -\delta_{n+1})$ with $\delta_n \in \{0,1\}$; 
the corresponding summand in $\EE$ is  a factor ring of $R = T/(x_n^2)$, namely
$$
\sigma\big(T(-\delta_1, \ldots, -\delta_n, -\delta_{n+1})\big) \cong \frac{R}{(x_n^{2-\delta_n})}(-\delta_1, \ldots, -\delta_n, -\delta_{n+1}).
$$
The cyclic $R$-module on the right hand side is free if and only if $\delta_n = 0$.
In fact, $\EE$ is an acyclic minimal $\mathbb{Z}^{n+1}$-graded complex of (not necessarily free) finitely generated $R$-modules.
Since all the $x_n$-twists in $\FF$  are in $\{0,1\}$,  every free summand of $\FF$ contributes with a  non-zero  summand in $\EE$.
In other words,
 in every homological degree $i$,
the numbers of generators is the same for $\FF$ and $\EE$, and this number is $\beta_i^{T}(\mI)$.
Among the direct summands of $\EE$, the free modules are precisely those coming from copies of $T$ in $\FF$ with $x_n$-twist equal to 0.
These modules form themselves another complex $\EE'$, which is again minimal and acyclic, but it is even free.
In fact, $\EE'$ is the minimal free resolution of $I_0$ over $\ovR$, since $I_0$ is the truncation of $\mI$ in $x_n$-degree $0$, and $\ovR$ is the truncation of $T$ in $x_n$-degree 0. 
We conclude that, in homological degree $i$, in $\EE$ we have exactly $\beta^\ovR_i(I_0)$  free summands, i.e.,  copies of $R$.

To summarize,
$\EE$ is an acyclic minimal  complex of $\ZZ^{n+1}$-graded  $R$-modules, 
it has $\beta_i^{T}(\mI)$  generators in homological degree $i$,
of which $\beta^\ovR_i(I_0)$   generate a free module $R$, 
whereas the remaining ones generate a non-free module isomorphic to $R/(x_n)$.
The number of non-free summands of $\EE$ in homological degree $i$ is,  
therefore,
$\beta_i^{T}(\mI)- \beta^\ovR_i(I_0)  =  r'_1\beta_i^{\ovR}(\Bbbk[x_{n+1}])$,
by \eqref{EqBettiNumbersTandOvR}.
Note also that the 0-homology of $\EE$ is $\sigma(T/\mI) =R/I$.

Let $E_i$ denote the module in homological degree $i$ in $\EE$.
The differentials of $\EE$ can be lifted to a complex of complexes, namely a double complex $\DD_I$ of $R$-modules where the 
$i$-th vertical complex is  the minimal free resolution of $E_i$.
By construction, the double complex  $\DD_I$ is free.
Furthermore, it is minimal, 
and the total complex $\mathrm{Tot}(\DD_I)$ is a minimal $\ZZ^{n+1}$-graded free resolution of $R/I$ over $R$,
cf. 
\cite[Proposition 5.6]{EiPoYu03},
\cite[Theorem 1.3]{ArAvHe00},
or
\cite[Theorem 2.10]{GaHiPe02}.
The $R$-module $R/(x_n)$  has an infinite minimal free resolution over $R$ with $\beta_j^R(R/(x_n))=1$ and differential given by $\cdot x_n$ for every $j \in \NN$.
It follows that in $\mathbb{D}_I$, for each $i\geq 0$,  we have 
\begin{itemize}
\item[$(\ast)$]$\beta^\ovR_i(I_0)$ summands in homological bidegree $(i,0)$ arising from the free summands of $\EE$,

\item[$(\ast \ast)$]  $r'_1\beta_i^{\ovR}(\Bbbk[x_{n+1}])$ summands in homological bidegree $(i,j)$ for every $j\in\NN$,
arising from the  non-free summands of $\EE$,
\end{itemize}
where the first coordinate is  horizontal and the second coordinate is vertical.
We conclude that the Betti numbers of a saturated strongly stable  $I\subseteq R$ depend only on those of $I_0\subseteq \ovR$ and on the number 
$r'_1 =\mult(I_0) - \mult(I_1)$. 

The same construction for $\CC(d)$ yields a double complex $\DD_{\CC(d)}$.
Let $r'_2 =  \mult(\CC(d)_0)-\mult(\CC(d)_1)$.
We observed in the proof of Theorem \ref{TheoremExtremalFiniteResolution} that $\mult(\CC(d)_{1})\leq \mult(I_{1}) $.
We deduce that $r'_2 \geq r_0+r'_1$.
Finally, 
we compare the contribution of the two types of summands $(\ast)$ and $(\ast\ast)$ to the double complexes
$\DD_I$ and $\DD_{\CC(d)}$:

\begin{itemize}
\item[$(\ast)$] 
For every $i\geq 0$,
by \eqref{InequalityBetti0InfiniteResolution}, 
$\DD_I$ has at most  $r_0\beta_i^{\ovR}(\Bbbk[x_{n+1}])$ more summands in position $(i,0)$ than  $\DD_{\CC(d)}$,
among those arising from the free summands of $\EE$.
\item[$(\ast\ast)$] 
For every $i,j\geq 0$,
$\mathbb{D}_{\CC(d)}$ has  at least  $(r'_2-r'_1)\beta_i^{\ovR}(\Bbbk[x_{n+1}])$ more summands in position $(i,j)$ than
$\mathbb{D}_I$, 
among those arising from the non-free summands of $\EE$.
\end{itemize}
Thus, 
$\mathbb{D}_{\CC(d)}$ has at least as many copies of $R$ as $\mathbb{D}_I$, in every position $(i,j)$.
This concludes the proof, 
since 
$\beta_i^R(I),\beta_i^R(\CC(d))$ 
are the Betti numbers of $\mathrm{Tot}(\mathbb{D}_I),\mathrm{Tot}(\mathbb{D}_{\CC(d)})$ respectively.
\end{proof}

In the rest of this section, 
we explore bounds for deviations and Poincar\'e series. 
The deviations of a ring $A$ are a  sequence of  integers $\{ \varepsilon_i(A)\}_{i \geq 1}$
measuring several homological or cohomological data of $A$.
Examples include: the  generators of a Tate resolution of $A$ over a polynomial ring, as well as a Tate resolution of $\Bbbk$ over $A$;
the ranks of the modules in a cotangent complex  of $A$;
 the dimensions of the  components of the homotopy Lie algebra $\pi(A)$ of $A$.
We refer to \cite[Sections 7 and 10]{Av98} for definitions and background.

\begin{lemma}\label{LemmaNoVariable}
Let $I \in \Hilb^d(\Proj R)$ be strongly stable. 
There is an inclusion of vector spaces of linear forms
$[\CC(d)]_1 \subseteq [I]_1$.
\end{lemma}
\begin{proof}
We may assume $d>0$.
Since $\CC(d)$ is saturated and strongly stable, we have $[\CC(d)]_1 = \langle x_1, \ldots, x_m\rangle_\Bbbk$ for some $0\leq m \leq n$.
If $m=n$, then $\CC(d) = (x_1, \ldots, x_n) \subseteq R$, so $d=1$ and  $I= \CC(d)$.
If $m<n$, then $[\CC(d)]_1 = [\CC(d)_0]_1$.
We induct on $n$, and the case $n=0 $ is trivial.
By  \hyperref[C4]{(C4)}, we have $\mult(I_0)\leq\mult(\CC(d)_0)$,
thus $\CC(d)_0\subseteq \CC(I_0)$ by \hyperref[C3]{(C3)}.
By induction,  $[\CC(I_0)]_1 \subseteq [I_0]_1$,
hence, $[\CC(d)]_1 = [\CC(d)_0]_1 \subseteq[\CC(I_0)]_1 \subseteq [I_0]_1 \subseteq  [I]_1$.
\end{proof}

A consequence of Theorem \ref{TheoremExtremalFiniteResolution} and the results of \cite{BoDaGrMoSa16}
is the fact that an $\CC(d)$  has maximal deviations in the Hilbert scheme of $\PP^n$.

\begin{cor}\label{CorollaryDeviations}
Let $S = \Bbbk[x_1, \ldots, x_{n+1}]$.
We have
$\varepsilon_{i}(S/I) \leq \varepsilon_{i}(S/\CC(d))
$ 
for every $I\in\Hilb^d(\PP^n)$ and  all $i \geq 1$.
\end{cor}
\begin{proof}
As  in the proof Lemma \ref{LemmaAlmostLexMerminMurai},
we may assume   $I:x_{n+1}=I$.
Let $\wL = \Lex(\wI) \subseteq \wS$.
By \cite[Theorem 3.4]{BoDaGrMoSa16},
we have $\varepsilon_{i}(\wS/\wI) \leq \varepsilon_{i}(\wS/\wL)$  for all $i \geq 2$.
It follows from \cite[Proposition 7.1.6]{Av98} that $\varepsilon_{i}(S/I) \leq \varepsilon_{i}(S/L)$  for all $i \geq 2$,
where $L = \wL S \subseteq S$.
The ideals $L$ and $\CC(d)$ are strongly stable, 
and this implies that $S/L$ and $S/\CC(d)$ are Golod rings
by \cite[Theorem 4]{HeReWe99}.
Now, by \cite[Proposition 3.2]{BoDaGrMoSa16}, 
we derive that $\varepsilon_{i}(S/L) \leq \varepsilon_{i}(S/\CC(d))$  for all $i \geq 2$.
Finally, for $i=1$,
the deviation $\varepsilon_1(A)$ is equal to the embedding dimension of $A$,
cf. \cite[Corollary 7.1.5]{Av98},
therefore,
$\varepsilon_{1}(S/L) \leq \varepsilon_{1}(S/\CC(d))$ by 
Lemma \ref{LemmaNoVariable}.
\end{proof}

In particular, $\CC(d)$  has maximal  Poincar\'e series, 
that is, the generating function of the dimensions of $\Tor_\bullet^A(\Bbbk,\Bbbk)$
or $\mathrm{Ext}^\bullet_A(\Bbbk,\Bbbk)$.

\begin{cor}\label{CorollaryPoincare}
Let $S = \Bbbk[x_1, \ldots, x_{n+1}]$.
We have
$\beta_{i}^{S/I}(\Bbbk) \leq \beta_{i}^{S/\CC(d))}(\Bbbk) $
for every $I\in\Hilb^d(\PP^n)$ and  all $i \geq 0$.
\end{cor}
\begin{proof}
Apply Corollary \ref{CorollaryDeviations} and \cite[Remark 7.1.1]{Av98}.
\end{proof}

We conclude this section  by proposing a  generalization of Corollaries \ref{CorollaryDeviations} and \ref{CorollaryPoincare}.

\begin{question}
Let $R$ be a Clements-Lindstr\"om ring. 
Is it true that
$\beta_{i}^{S/I}(\Bbbk) \leq \beta_{i}^{S/\CC(d))}(\Bbbk) $
and
$\varepsilon_{i}(S/I) \leq \varepsilon_{i}(S/\CC(d))
$ 
for every $I\in \Hilb^d(\Proj R)$ and  every $i \geq 0$?
\end{question}

\section{Applications and examples}\label{SectionApplications}

We conclude the paper  by illustrating the applications of our results to Hilbert schemes of points of arbitrary complete intersections,
and by exhibiting explicit examples of the numerical bounds obtained from 
Theorem \ref{TheoremExtremalFiniteResolution}.

\begin{example}
Consider the Clements-Lindstr\"om ring $R = \Bbbk[x_1, x_2, x_3, x_4]/(x_1^2, x_2^2)$
and the Hilbert scheme  $\Hilb^{20}(\Proj R)$.
In order to apply Theorem \ref{TheoremExtremalFiniteResolution},
we compute the Betti numbers of  $\CC(20,R) = (x_1,x_2,x_3)^6 =  
( x_1x_2x_3^4, x_1x_3^5, x_2x_3^5, x_3^6) \subseteq R $,
and find the sharp upper bounds for the syzygies of  $I \in \Hilb^{20}(\Proj R)$
$$
\beta_1^S(R/I) \leq 6,\,
\beta_2^S(R/I) \leq 9, \,
\beta_3^S(R/I) \leq 4.
$$
Note that $\Proj R \subseteq \PP^3$.
If we instead regard $I$ as an element of $\Hilb^{20}(\PP^3)$, 
and use the results of \cite{CaMu13} or \cite{Va94}, 
which involve the Betti numbers of 
$\CC(20,S) = (x_1, x_2, x_3)^4\subseteq S$,
we find the coarser bounds
$$
\beta_1^S(R/I) \leq 15,\,
\beta_2^S(R/I) \leq 24, \,
\beta_3^S(R/I) \leq 10.
$$
\end{example}

We say that a regular sequence $f_1, \ldots, f_c$ has degree sequence $e_1 \leq e_2 \leq \cdots \leq e_n \leq \infty$ if
$c=\max\{j \, : e_j < \infty\}$ and   $e_i = \deg(f_i)$ for every $i \leq c$,
and we extend the same terminology to complete intersections $X \subseteq \mathbb{P}^n$.
A notable consequence of Theorem \ref{TheoremExtremalFiniteResolution} is the fact that, conjecturally, 
it provides sharp upper bounds for all  subschemes $Z\in  \Hilb^d(X)$ of \emph{all} complete intersections $X  \subseteq \mathbb{P}^n$
 with a given degree sequence.
 To justify this claim, 
we recall two famous conjectures on complete intersections.
For our purposes, 
it is convenient to state them in terms of ideals of $\wS = \Bbbk[x_1, \ldots, x_n]$.
 
 \begin{conj}[Eisenbud-Green-Harris]\label{ConjEGH}
If $I\subseteq \wS$ contains a regular sequence of degree sequence  $e_1, \ldots, e_n$, 
then there exists a lex ideal $L\subseteq \wS$ with $\HF(I) = \HF \big( L + (x_1^{e_1},\ldots, x_n^{e_n})\big)$.
\end{conj}

\begin{conj}[Lex-Plus-Powers]\label{ConjLPP}
If $I\subseteq \wS$ contains a regular sequence of degree sequence  $e_1, \ldots, e_n$ and if
there exists a lex ideal $L\subseteq \wS$ with $\HF(I) = \HF \big( L + (x_1^{e_1},\ldots, x_n^{e_n})\big)$,
then $\beta_{i,j}^\wS(\wS/I) \leq \beta_{i,j}^\wS\big(\wS/(L + (x_1^{e_1},\ldots, x_n^{e_n}))\big) $
for all $i,j$.
\end{conj}

We  refer to them as the EGH and LPP Conjectures.
Despite the apparently independent statements, Conjecture \ref{ConjLPP} actually implies Conjecture \ref{ConjEGH}:
more precisely, 
the EGH Conjecture is equivalent to the statement of the LPP Conjecture for $i=1$, see for example \cite[Conjecture 4.7]{FrRi07} and the discussion preceding it.
We refer to   \cite{CDSS21,FrRi07,Gu21} for an overview of these two problems.
We denote $\mu(Z) = \beta_1^S(S/I_Z)$, the number of generators of the saturated ideal  $I_Z\subseteq S$ of a  closed subscheme $Z\subseteq \PP^n$.

\begin{prop}\label{PropositionEGHLPP}
Let $X  \subseteq \PP^n$ be a complete intersection 
of degree sequence $e_1 \leq \cdots \leq e_n \leq \infty$,
and consider the Clements-Lindstr\"om ring $R=\Bbbk[x_1, \ldots, x_{n+1}]/\big(x_1^{e_1},\ldots, x_n^{e_n}\big)$.
\begin{enumerate}

\item If the EGH Conjecture holds, then 
$
\mu(Z) \leq \beta_1^S(R/\CC(d,R)) 
$
for every  $Z\in\Hilb^d(X)$.

\item If the LPP Conjecture holds, then 
$
\beta_{i}^S\big(S/I_Z\big) \leq \beta_{i}^S\big(R/\CC(d,R)\big) 
$
for every  $Z\in\Hilb^d(X)$ and  every $i = 0,\ldots, n$.
\end{enumerate}
\end{prop}
\begin{proof}
It suffices to present the proof for (2).
As in the proof of Lemma \ref{LemmaAlmostLexMerminMurai},
we may assume that $x_{n+1}$ is a non-zerodivisor on $S/I_Z$, and we 
consider $\wI = \frac{I_Z+(x_{n+1})}{(x_{n+1})}\subseteq \wS$.
By assumption, both Conjectures \ref{ConjEGH} and \ref{ConjLPP} hold,
so,  
there  exists a lex ideal $\wL \subseteq \wR$ such that $\HF(\wS/\wI) =\HF(\wR/\wL) $
and $\beta_{i,j}^\wS(\wS/\wI) \leq \beta_{i,j}^\wS(\wR/\wL) $ for all $i,j\geq 0$.
The ideal $L= \wL R \subseteq R$ is  almost lex.
By Remark \ref{RemarkModuloXNplus1}, we have $\beta_{i,j}^S(S/I_Z) \leq \beta_{i,j}^S(R/L) $ and $\HF(R/L)=\HF(S/I_Z)$,
hence, we have $\mult(L)= \mult(I_Z)=d$.
The conclusion follows from  Theorem \ref{TheoremExtremalFiniteResolution}.
\end{proof}

The EGH Conjecture has been proved in several cases, cf. \cite{CDSS21,Gu21}.
Here we sample some of the possible applications of Proposition \ref{PropositionEGHLPP}.

\begin{example}
Let $X \subseteq \PP^7$ be a complete intersection of 5 quadrics,
and   $Z\subseteq X$  a finite subscheme of length 60.
The EGH Conjecture holds for $X$ by  \cite{GH20}.
We compute $\CC(60,R)\subseteq R = \Bbbk[x_1, \ldots, x_8]/(x_1^2, \ldots, x_5^2)$,
and deduce
$
\mu(Z) \leq 66$
 by Proposition \ref{PropositionEGHLPP}.
\end{example}

\begin{example}
Let $X \subseteq \PP^5$ be a complete intersection of 3 cubics,
and   $Z\subseteq X$  a finite subscheme of length 60.
The EGH Conjecture holds for $X$ by  \cite{CaDS20}.
We compute $\CC(60,R)\subseteq R = \Bbbk[x_1, \ldots, x_6]/(x_1^3, x_2^3, x_3^3)$,
and deduce
$
\mu(Z) \leq 59$
 by Proposition \ref{PropositionEGHLPP}.
\end{example}

On the other hand, Conjecture \ref{ConjLPP} is known in very few cases.
Using \cite[Main Theorem]{CaSa18} and Proposition \ref{PropositionEGHLPP},
 we obtain the following result.

\begin{cor}\label{CorollaryLPP}
Assume   $\ch(\Bbbk) = 0$. 
Let $X  \subseteq \PP^n$ be a complete intersection 
with degree sequence 
such that  $e_j > \sum_{h=1}^{j-1}(e_h-1)$ for  $j\geq 3$.
Then, 
$
\beta_{i}^S\big(S/I_Z\big) \leq \beta_{i}^S\big(R/\CC(d,R)\big) 
$
for all  $Z\in\Hilb^d(X)$ and   $i \geq 0$.
\end{cor}

\begin{example}
An elliptic quartic $C\subseteq \PP_{\mathbb{C}}^3$ is a complete intersection of 2 quadric surfaces. 
Every 0-dimensional scheme $Z$ lying on $C$ satisfies
$$
\beta_1^S(S/I_Z) \leq 6,\,
\beta_2^S(S/I_Z) \leq 9, \,
\beta_3^S(S/I_Z) \leq 4.
$$
To see this, let $R = \mathbb{C}[x_1,x_2,x_3,x_4]/(x_1^2,x_2^2)$.
The ideals $\CC(d,R)$, 
with $d \in \NN$, are 
$(x_1,x_2,x_3)$, $(x_1,x_2,x_3^2)$, $(x_1,x_2x_3,x_3^2),$ or 
$(x_1x_2x_3^\alpha, x_1x_3^{\alpha+1+\delta_1}, x_2x_3^{\alpha+1+\delta_2}, x_3^{\alpha+2+\delta_3})$
for some  $\alpha\in \NN$ and $0\leq \delta_1 \leq \delta_2 \leq \delta_3\leq 1$.
The claimed bounds follow  from Corollary \ref{CorollaryLPP},
by calculating the Betti numbers in all four cases.
\end{example}

\begin{example}
Let $X \subseteq \PP_{\mathbb{C}}^4$ be a complete intersection of 3 quadrics
and let $Z \subseteq X$  be a finite subscheme of length 60.
While Corollary \ref{CorollaryLPP} does not apply directly to the degree sequence $(2,2,2)$, 
it can still be used to provide upper bounds
that are sharper than the general ones valid for $\Hilb^{60}(\PP^4)$,
 arguing as in \cite[Example 4.3 and Remark 4.4]{CaSa18}.
In fact, any ideal containing $I_X$ also contains a regular sequence of degrees 
$e_1 = 2, e_2= 2, e_3 =3$.
Therefore, letting $R = \mathbb{C}[x_1,x_2,x_3,x_4,x_5]/(x_1^2,x_2^2, x_3^3)$
and determining $\CC(60,R)$, Corollary \ref{CorollaryLPP} yields
$$
\beta_1^S(S/I_Z) \leq 15,\,
\beta_2^S(S/I_Z) \leq 39, \,
\beta_3^S(S/I_Z) \leq 37, \,
\beta_4^S(S/I_Z) \leq 12.
$$
\end{example}

\subsection*{Acknowledgments}
The  authors would like to thank  Paolo Lella and Roberto Notari for pointing out an error in a previous version of this manuscript,
and Ritvik Ramkumar for some helpful conversations.
Computations with Macaulay2 \cite{M2} provided valuable insights during the preparation of this paper.

\end{document}